\documentclass[letterpaper, 10 pt, conference]{ieeeconf}  

\IEEEoverridecommandlockouts                              

\usepackage{times}
\usepackage{graphicx} 

\usepackage{algorithm}
\usepackage{algorithmic}

\usepackage{todonotes}

\usepackage{fullpage,psfrag,amsthm,amsmath,amsfonts,verbatim,mathtools,subfig,enumerate,bm,hyperref}
\allowdisplaybreaks

\newcommand{\reals}{{\mbox{\bf R}}}

\newcommand{\argmin}{\mathop{\rm argmin}}

\newcommand{\norm}[1]{\left\lVert#1\right\rVert}
\newcommand{\mnorm}[1]{{\left\vert\kern-0.25ex\left\vert\kern-0.25ex\left\vert #1 
    \right\vert\kern-0.25ex\right\vert\kern-0.25ex\right\vert}}

\newtheorem{theorem}{Theorem}
\newtheorem{lemma}{Lemma}
\newtheorem{corollary}{Corollary}
\newtheorem{remark}{Remark}
\newtheorem{proposition}{Proposition}
\newtheorem{assumption}{Assumption}

\newcommand{\ie}{{\it i.e.}}

\title{\LARGE \bf
Bregman Parallel Direction Method of Multipliers for Distributed Optimization via Mirror Averaging
}

\author{Yue Yu, Beh\c{c}et A\c{c}\i kme\c{s}e and Mehran Mesbahi
\thanks{Accepted to IEEE Control Systems Letters (L-CSS).}%
}

\begin{document}

\maketitle
\thispagestyle{empty}
\pagestyle{empty}

\begin{abstract}

Distributed optimization aims to optimize a global objective formed by a sum of coupled local convex functions over a graph via only local computation and communication. In this paper, we propose the Bregman parallel direction method of multipliers (PDMM) based on a generalized averaging step named mirror averaging. We establish the global convergence and $O(1/T)$ convergence rate of the Bregman PDMM, along with its $O(n/\ln n)$ improvement over existing PDMM, where $T$ denotes the number of iterations and $n$ the dimension of solution variable. In addition, we can enhance its performance by optimizing the spectral gap of the averaging matrix. We demonstrate our results via a numerical example. 

\end{abstract}

\section{Introduction}
Distributed optimization arises in a variety of applications such as distributed tracking and localization \cite{li2002detection}, estimation in sensor networks \cite{lesser2012distributed}, and multiagent coordination \cite{xiao2007distributed}. In particular, given an undirected connected graph with \(m\) vertices, distributed optimization over this graph is defined as
\begin{equation}
\begin{array}{ll}
\underset{u\in\mathcal{X}}{\mbox{minimize}} & \sum_{i=1}^m f_i(u)
\end{array}
\label{distributed optimization problem}
\end{equation}
where \(\mathcal{X}\subseteq\reals^n\) is a closed convex set, each \(f_i\) is a convex function locally known by vertex \(i\) only. The optimality is achieved by local optimization on each vertex and efficient communication between  neighboring vertices in the graph.

Alternating direction method of multipliers (ADMM) \cite{boyd2011distributed} is a primal-dual algorithm that alternatively optimizes a quadratic augmented Lagrangian with respect to splitted primal variables and dual variables. There has been an increasing interest in applying multi-block variants of ADMM to solve problem \eqref{distributed optimization problem} \cite{he2012alternating, deng2017parallel, hong2017linear}. One of the main challenges of such methods is to find a separable approximation to the coupled quadratic penalty term in augmented Lagrangian. In particular, a Gauss-Seidel approximation \cite{he2012alternating, hong2017linear} was proposed in \cite{wei2012distributed}, which results in sequential updates on the vertices. On the other hand, a Jacobian approximation based variant of ADMM \cite{jakovetic2013distributed, deng2017parallel} allows simultaneous updates \cite{wang2014parallel, meng2015proximal}. We call such methods parallel direction method of multipliers (PDMM) since their primal variables are updated parallelly instead of alternatively.

Bregman ADMM \cite{wang2014bregman} is a generalization of ADMM where the quadratic penalty function in ADMM updates is replaced by Bregman divergence, which can potentially exploit the problem structure. There has been attempts to introduce Bregman divergence as proximal term in multi-block variants of ADMM \cite{wang2014parallel, wang2015convergence}, but they are still based on a quadratic augmented Lagrangian. To our best knowledge, all existing ADMM based methods for distributed optimization use quadratic penalty functions. 

In this paper, we propose a new solution method, namely Bregman PDMM, for distributed optimization, which combines the advantages of PDMM and Bregman ADMM. We first propose a generalized averaging step named mirror averaging. Based on this, we develop Bregman PDMM which replaces all the quadratic penalty function in PDMM updates with Bregman divergence. We establish the global convergence of the proposed algorithm and its \(O(1/T)\) convergence rate, where \(T\) is the number of iterations. Furthermore, in some cases, Bregman PDMM can outperform PDMM by a factor of \(O(n/\ln n)\), where \(n\) is the dimension of solution variable. Finally, we show that by optimizing the spectral gap of the averaging matrix, we can enhance the performance of the Bregman PDMM. 

The rest of the paper is organized as follows. \S\ref{preliminaries} provides a reformulation of problem \eqref{distributed optimization problem} using consensus constraints. In \S\ref{method}, we develop Bregman PDMM for problem \eqref{distributed optimization problem}, whose convergence properties are established in \S\ref{convergence} via Lyapunov analysis. \S\ref{numerical examples} presents numerical examples; \S\ref{conclusion} concludes the paper and comments on the future work.  
\section{Preliminaries and Background}\label{preliminaries}

\subsection{Notation}
Let \(\reals\) (\(\reals_+\)) denote the (nonnegative) real numbers, \(\reals^n\) (\(\reals^n_+\))  denote the set of \(n\)-dimensional (elementwise nonnegative) vectors.  Let \(\geq(\leq)\) denote elementwise inequality when applied to vectors and matrices. Let \(\langle\cdot, \cdot\rangle\) denote the dot product.  Let \(I_n\in\reals^{n\times n}\) denote the \(n\)-dimensional identity matrix, \(\mathbf{1}_n\in\reals^n\) the \(n\)-dimensional vector of all \(1\)s. Given matrix \(A\in\reals^{n\times n}\), let \(A_{ij}\) denote its \((i, j)\) entry; \(A^\top\) denotes its transpose. Let \(\otimes\) denote the Kronecker product. Given the set \(\mathcal{X}\subseteq\reals^n\),  its indicator function \(\iota_\mathcal{X}:\reals^n\to \reals\) is defined as: \(\iota_\mathcal{X}(u)=0\) if \(u\in\mathcal{X}\) and \(+\infty\) otherwise.
\subsection{Subgradients}
Let \(f:\reals^n\to\reals\) be a convex function. Then \(g\in\reals^n\) is a subgradient of \(f\) at \(u\in\reals^n\) if and only if for any \(v\in\reals^n\) one has
\begin{equation}
f(v)-f(u)\geq \left\langle g, v-u\right\rangle.
\label{definition of subgradient}
\end{equation}
We denote \(\partial f(u)\) the set of subgradients of \(f\) at \(u\). We will use the following results.
\begin{lemma}\cite[Theorem 27.4]{rockafellar2015convex} Given a closed convex set \(\mathcal{C}\subseteq\reals^n\) and closed, convex, proper function \(f:\reals^n\to \reals\), then \(u^\star=\underset{u\in\mathcal{C}}{\argmin}\, f(u)\) if and only if there exists \(g\in N_\mathcal{C}(u^\star)\) such that \(-g\in\partial f(u^\star)\), where 
\begin{equation}
N_\mathcal{C}(u^\star)\coloneqq\left\{g\in\reals^n: \langle g, u^\star-v\rangle\geq 0, \forall v\in\mathcal{C}\right\}
\label{definition of normal cone}
\end{equation}
\label{lemma optimality condition} 
is the normal cone of the set \(\mathcal{C}\) at \(u^\star\).
\end{lemma} 
\subsection{Mirror maps and Bregman divergence} 
Let \(\mathcal{D}\subseteq\reals^n\) be a convex open set. We say that \(\phi:\mathcal{D}\to\reals\) is a {\em mirror map} \cite[p.298]{bubeck2015convex} if it satisfies: 1) \(\phi\) is differentiable and strictly convex, 2) \(\nabla\phi\) takes all possible values, and 3) \(\nabla\phi\) diverges on the boundary of the closure of \(\mathcal{D}\), \ie , \(\lim_{u\to\partial \bar{\mathcal{D}}}\norm{\nabla\phi(u)}=\infty\), where \(\norm{\cdot}\) is an arbitrary norm on \(\reals^n\). Bregman divergence \(B_\phi:\mathcal{D}\times\mathcal{D}\to\reals_+\) induced by \(\phi\) is defined as
\begin{equation}
B_\phi(u, v)=\phi(u)-\phi(v)-\left\langle \nabla\phi(v), u-v\right\rangle.
\label{definition of Bregman divergence}
\end{equation}
Note that \(B_\phi(u, v)\geq 0\) and \(B_\phi(u, v)=0\) only if \(u=v\). $\Phi$ and $B_\phi$ also satisfy the following three-point identity,
\begin{equation}
\begin{aligned}
&\langle\nabla\phi(u)-\nabla\phi(v), w-u\rangle\\
=&B_\phi(w, v)-B_\phi(w, u)-B_\phi(u, v).
\end{aligned}\label{3-point property}
\end{equation}  

\subsection{Graphs and distibuted optimization}
An undirected connected graph \(\mathcal{G}=(\mathcal{V}, \mathcal{E})\) contains a vertex set \(\mathcal{V}=\{1, 2, \ldots, m\}\) and an edge set \(\mathcal{E}\subseteq \mathcal{V}\times \mathcal{V}\) such that \((i, j)\in\mathcal{E}\) if and only if \((j, i)\in\mathcal{E}\) for all \(i, j\in\mathcal{V}\). Denote \(\mathcal{N}(i)\) the set of neighbors of node \(i\), where \(j\in\mathcal{N}(i)\) if \((i, j)\in\mathcal{E}\). 

Consider a symmetric stochastic matrix \(P\in\reals^{m\times m}\) defined on the graph \(\mathcal{G}\) such that \(P_{ij}>0\) implies that \(j\in\mathcal{N}(i)\). Such a matrix \(P\) can be constructed, for example, by the graph Laplacian \cite[Proposition 3.18]{mesbahi2010graph}. The eigenvalues of \(P\) are real and will be ordered nonincreasingly in their magnitude, denoted by \(|\lambda_1(P)|\geq |\lambda_2(P)|\geq \ldots\geq|\lambda_m(P)|\). From \cite[Theorem 8.4.4]{horn2012matrix} we know that \(\lambda_1(P)=1\) is simple with eigenvectors spanned by \(\mathbf{1}_m\).


Let \(\mathcal{G}=(\mathcal{V}, \mathcal{E})\) denote the underlying graph over which the distributed optimization problem \eqref{distributed optimization problem} is defined. A common approach  to solve problem \eqref{distributed optimization problem} is to create local copies of the design variable \(\{x_1, x_2, \ldots, x_m\}\) and impose the consensus constraints: \(x_i=x_j\) for all \((i, j)\in \mathcal{E}\) \cite{bertsekas1989parallel, boyd2011distributed}. Many different forms of consensus constraints have been proposed \cite{wei2012distributed, jakovetic2013distributed, iutzeler2013asynchronous, shi2014linear}. In this paper, we consider consensus constraints of the form:
\begin{equation}
(P\otimes I_n)\bm{x}=\bm{x},
\label{consensus constraints}
\end{equation}
where \(\bm{x} =[x^\top_1, x^\top_2, \ldots, x^\top_m]^\top\), \(P\) is a symmetric, stochastic and irreducible matrix defined on \(\mathcal{G}\). We will focus on solving the following reformulation of \eqref{distributed optimization problem},
\begin{equation}
\begin{array}{ll}
\underset{\bm{x}\in\mathcal{X}^m}{\mbox{minimize}} & \sum_{i\in\mathcal{V}} f_i(x_i)\\
\mbox{subject to} & (P\otimes I_n) \bm{x}= \bm{x},
\end{array}
\label{consensus optimization problem}
\end{equation}
where \(\mathcal{X}^m\) is the Cartesian product of \(m\) copies of \(\mathcal{X}\).
\section{Bregman Parallel Direction Method of Multipliers}\label{method}
In this section, we first introduce an existing PDMM that contains averaging as an implicit update. Then we generalize averaging to mirror averaging based on the idea of mirror descent, and finally propose our Bregman PDMM based on mirror averaging.

\subsection{Parallel Direction Method of Multipliers} 
PDMM \cite{meng2015proximal} solves \eqref{distributed optimization problem} with \(\mathcal{X}=\reals^n\) with parallel and single loop primal updates, and links convergence behavior to graph topology \cite{makhdoumi2017convergence, francca2017distributed}.  An adaption of PDMM to formulation \eqref{consensus optimization problem} is given in Algorithm~\ref{PDMM}.

Naturally, one will try to generalize the quadratic penalty in Algorithm~\ref{PDMM} to Bregman divergence the same way  Bregman ADMM generalizes ADMM \cite{wang2014bregman}. However, if we simply replace the quadratic penalty in Algorithm~\ref{PDMM} with Bregman divergence induced by a strongly convex function \(\phi\), it is challenging  to prove its convergence for the following reasons. A crucial step in the proof provided in \cite{meng2015proximal} is to apply the three point identity \eqref{3-point property} to a convex function \(\Psi:\reals^{mn}\to\reals\) that satisfies the following differential equation,
\begin{equation}
\nabla \Psi(\bm{u})=(P\otimes I_n)\nabla\Phi(\bm{u}),
\label{Dennis ODE}
\end{equation}
where \(\Phi(\bm{u})=\sum_{i\in\mathcal{V}}\phi(u_i)\) with \(\bm{u} = [u_1^\top, \ldots, u_m^\top]^\top\in\mathcal{X}^m\). However, it is highly non-trivial to solve \eqref{Dennis ODE} for a convex function \(\Psi\) unless \(\phi\) is quadratic function. Hence we cannot directly utilize the convergence proof in \cite{meng2015proximal}. Therefore, we need to take a closer look at the role of the quadratic term in \eqref{PDMM: primal update}.

\begin{algorithm}
\caption{Existing PDMM \cite{meng2015proximal}}
\begin{algorithmic}[h]
\REQUIRE Parameter \(\rho>0\); initial point \(\bm{x}^{(0)}, \bm{\nu}^{(0)}\in\reals^{mn}\).
\FORALL{\(t=0, 1, 2, \ldots\)}
\STATE{each vertex \(i\) updates \(x_i\) in parallel
\begin{equation}
\begin{aligned}
x_{i}^{(t+1)} =  & \underset{x_i}{\argmin}  \,\,  f_i(x_i)\\
& +\langle x_i, \nu_i^{(t)} - \sum_{j\in\mathcal{N}(i)}P_{ij}\nu^{(t)}_j\rangle\\
& + \frac{\rho}{2} \sum_{j\in\mathcal{N}(i)} P_{ij}\norm{x_i - x_j^{(t)}}_2^2
\end{aligned}\label{PDMM: primal update}
\end{equation}
each vertex \(i\) updates \(\nu_i\)
\begin{equation}
\nu_i^{(t+1)} =  \nu_i^{(t)}+\rho x_i^{(t+1)}-\rho\sum_{j\in\mathcal{N}(i)}P_{ij}x_j^{(t+1)}  
\label{PDMM: dual update}
\end{equation}}
\ENDFOR
\end{algorithmic}
\label{PDMM}
\end{algorithm}

Consider the following intermediate update,
\begin{equation}
y_i^{(t)}\coloneqq\underset{x_i}{\argmin}\sum_{j\in\mathcal{N}(i)} P_{ij}\norm{x_i-x_j^{(t)}}_2^2.
\label{averaging}
\end{equation}
The behavior of \eqref{averaging} is characterized by the Markov chain defined by matrix \(P\) \cite[Proposition 3.21]{mesbahi2010graph}. In the sequel, we will generalize the quadratic function in \eqref{averaging} to Bregman divergence; then we will introduce Bregman PDMM based on such a generalization.

\subsection{Mirror Averaging}
Consider the following update: for all \(i\in\mathcal{V}\),
\begin{equation}
y_i^{(t)} = \underset{x_i\in \mathcal{X}}{\argmin}\sum_{j\in\mathcal{N}(i)} P_{ij} B_\phi(x_i, x_j^{(t)}),
\label{mirror averaging dynamics}
\end{equation}
where \(P\) is symmetric, stochastic and irreducible matrix defined on \(\mathcal{G}\), and \(\phi\) is a mirror map defined on the open set \(\mathcal{D}\) such that \(\mathcal{X}\) is included in the closure of \(\mathcal{D}\). 

Let \(\Phi(\bm{u})=\sum_{i\in\mathcal{V}} \phi(u_i)\) with \(\bm{u}=[u_1^\top, \ldots, u_n^\top]^\top\). Using an argument similar to the one in \cite[p. 301]{bubeck2015convex}, one can obtain the following result.
\begin{proposition}
Update \eqref{mirror averaging dynamics} is equivalent to
\begin{subequations}
\begin{align}
\nabla\Phi (\bm{z}^{(t)}) &= (P\otimes I_n)\nabla \Phi(\bm{x}^{(t)}), \label{averaging in mirror}\\
\bm{y}^{(t)}  &=  \underset{\bm{x}\in \mathcal{X}^m}{\argmin}\,\, B_\Phi(\bm{x}, \bm{z}^{(t)}).\label{projection in mirror}
\end{align}
\label{mirror dynamics break down}
\end{subequations}
\label{proposition 1}
\end{proposition}
Since \eqref{averaging in mirror} has the same dynamics as averaging step \eqref{averaging}, inspired by the idea of mirror descent, we interpret \eqref{mirror averaging dynamics} as \emph{mirror averaging}: to achieve update \eqref{mirror averaging dynamics}, we first map \(\bm{x}^{(t)}\) to \(\nabla\Phi(\bm{x}^{(t)})\), next run an averaging step via \eqref{averaging in mirror} and obtain \(\nabla\Phi(\bm{z}^{(t)})\), then apply \(\left(\nabla\Phi\right)^{-1}\) to it and obtain \(\bm{z}^{(t)}\), and finally get \(\bm{y}^{(t)}\) via the projection \eqref{projection in mirror}.   

\begin{remark} We provide two special cases \cite[p~301]{bubeck2015convex} where \eqref{mirror averaging dynamics} has a close form solution:1) If \(\mathcal{X}=\reals^n\) and \(\phi=\norm{\cdot}_2^2\), then \eqref{averaging in mirror} and \eqref{projection in mirror} reduces to \eqref{averaging}. 2) If \(\mathcal{X}\) denotes the probability simplex and \(\phi\) the negative entropy function, then \eqref{averaging in mirror} reduces to weighted geometric averaging and \eqref{projection in mirror} to a simple re-normalization. 
\end{remark}

\begin{algorithm}
\caption{Bregman PDMM}
\begin{algorithmic}[h]
\REQUIRE Parameters: \(\tau, \rho >0\), \(\delta_1, \ldots, \delta_m\geq 0\); initial point \(\bm{x}^{(0)}\in\mathcal{X}^m\cap\mathcal{D}^m , \bm{\nu}^{(0)}\in\reals^{mn}\).
\FORALL{\(t=0, 1, 2, \ldots\)}
\STATE{each vertex \(i\) updates \(y_i\) and \(x_i\) in parallel
\begin{subequations}
\begin{align}
&y_i^{(t)}  =  \underset{y_i\in \mathcal{X}}{\argmin}\sum_{j\in\mathcal{N}(i)} P_{ij} B_\phi(y_i, x_j^{(t)})\label{Bregman PDMM: mirror Markov update}\\
&\begin{aligned}
x_{i}^{(t+1)} = \,&  \underset{x_i\in\mathcal{X}}{\argmin}  \,\,  f_i(x_i)\\
& +\langle x_i, \nu_i^{(t)} - \sum_{j\in\mathcal{N}(i)}P_{ij}\nu^{(t)}_j\rangle \\
& + \rho B_\phi(x_i, y_i^{(t)})+\delta_iB_{\varphi_i}(x_i, x_i^{(t)})\label{Bregman PDMM: primal update}
\end{aligned}
\end{align}
\label{Bregman PDMM: subproblem}
\end{subequations}
each vertex \(i\) updates \(\nu_i\) in parallel
\begin{equation}
\nu_i^{(t+1)} =  \nu_i^{(t)}+\tau x_i^{(t+1)}-\tau\sum_{j\in\mathcal{N}(i)}P_{ij}x_j^{(t+1)} \label{Bregman PDMM: dual update} 
\end{equation}}
\ENDFOR
\end{algorithmic}
\label{Bregman PDMM}
\end{algorithm}

We introduce the following useful lemma, whose proof can be found in the Appendix.
\begin{lemma}
Given update \eqref{mirror averaging dynamics}, for any \(u\in\mathcal{X}\), 
\begin{equation}
\begin{aligned}
&\sum_{i\in\mathcal{V}} B_\phi(u, x_i^{(t)})-\sum_{i\in\mathcal{V}} B_\phi(u, y_i^{(t)})\\
\geq&\sum_{i,j\in\mathcal{V}}P_{ij}B_\phi(y_i^{(t)}, x_j^{(t)}).
\end{aligned}
\label{lemma Pythagorean: eqn1}
\end{equation}\label{lemma Pythagorean}
\end{lemma}
\begin{remark}
Lemma~\ref{lemma Pythagorean} turns out to be a key step in our convergence proof. Notice that without the generalization from \eqref{averaging} to \eqref{mirror averaging dynamics}, we can replace Lemma~\ref{lemma Pythagorean} with Jensen's inequality for strongly convex function by assuming the Bregman divergence is strongly convex in the second argument, but such an assumption does not hold in general. Hence the generalization from \eqref{averaging} to \eqref{mirror averaging dynamics} is necessary.
\end{remark}

\subsection{Bregman PDMM via Mirror Averaging}
Based on the above observations, we propose Algorithm~\ref{Bregman PDMM} by generalizing the quadratic penalty term in Algorithm~\ref{PDMM} to Bregman divergence. It essentially combines the parallel updates in Algorithm~\ref{PDMM} and Bregman penalty term in Bregman ADMM \cite{wang2014bregman}. Notice that Algorithm~\ref{PDMM} is a special case of Algorithm~\ref{Bregman PDMM} with \(\phi=\frac{1}{2}\norm{\cdot}_2^2\), \(\tau=\rho\), and \(\delta_i=0\) for all \(i\in\mathcal{V}\). 

\section{Convergence}\label{convergence}
In this section, we establish the convergence analysis of Algorithm~\ref{Bregman PDMM}. All detailed proofs in this section can be found in the Appendix. We first define the Lagrangian of problem \eqref{consensus optimization problem} as \(L(\bm{x}, \bm{\nu})=\sum_{i\in\mathcal{V}} L_i(x_i, \bm{\nu})\) where, 
\begin{equation}
L_i(x_i, \bm{\nu})\coloneqq  f_i(x_i) + \iota_{\mathcal{X}}(x_i)+\langle x_i, \nu_i-\sum_{j\in\mathcal{V}}P_{ij}\nu_j\rangle,
\label{Lagrangian}
\end{equation}
and \(\bm{\nu}=[\nu_1^\top, \ldots, \nu_m^\top]^\top\) denote the dual variables. 

We group our assumptions in Assumption~\ref{basic assumption}.
\begin{assumption}
\label{basic assumption}
\begin{enumerate}[(a)]
\item For all \(i\in\mathcal{V}\), \(f_i:\reals^n\to\reals\cup \{+\infty\}\) is closed, proper and convex.
\item There exists a saddle point \((\bm{x}^\star, \bm{\nu}^\star)\) that satisfies the KKT conditions of the Lagrangian given in \eqref{Lagrangian}: for all \(i\in\mathcal{V}\), there exists \(g_i\in N_\mathcal{X}(x_i^\star)\) such that
\begin{subequations}
    \begin{align}
      \sum_{j\in\mathcal{V}}P_{ij}x_j^\star & =x_i^\star \label{KKT: primal} \\
      -\nu^\star_i+\sum_{j\in\mathcal{V}} P_{ij}\nu^\star_j -g_i& \in \partial f_i(x_i^\star) \label{KKT: dual}
    \end{align} \label{KKT}
\end{subequations}
\item Functions \(\varphi_1, \varphi_2, \ldots, \varphi_n:\mathcal{D}\to\reals\) are strictly convex, where \(\mathcal{D}\) is a open convex set such that \(\mathcal{X}\) is included in its closure. Function \(\phi:\mathcal{D}\to\reals\) is a mirror map and is \(\mu\)-strongly convex with respect to \(l_p\)-norm \(\norm{\cdot}_p\) over \(\mathcal{X}\cap\mathcal{D}\), \ie, for any \(u, v\in \mathcal{X}\cap\mathcal{D}\), 
\begin{equation}
B_\phi(u, v)\geq \frac{\mu}{2}\norm{u-v}_p^2.
\label{phi strong convexity}
\end{equation} 
\item The matrix \(P\) is symmetric, stochastic, irreducible and positive semi-definite. \label{P matrix assumption}
\end{enumerate}
\end{assumption}
\begin{remark}
An immediate implication of assumptions in entry~(\ref{P matrix assumption}) is that \(\lambda_2(P)<1\) due to  \cite[Corollary 8.4.6]{horn2012matrix}.    
\end{remark}

Notice that we assume a homogeneous mirror map \(\phi\) is used by all vertices in Algorithm~\ref{Bregman PDMM}, but our results can be generalized to the cases of heterogeneous mirror maps as long as they all satisfy \eqref{phi strong convexity}. 


Now we start to construct the convergence proof of Algorithm~\ref{Bregman PDMM} under Assumption~\ref{basic assumption}. From the definition in \eqref{Lagrangian} we know that the Lagrangian \(L(\bm{x}, \bm{\nu})\)  is separable in each \(x_i\); hence the KKT conditions \eqref{KKT: dual} can be obtained separately for each \(x_i\) using Lemma~\ref{lemma optimality condition}. 
Similarly one can have the optimality condition of \eqref{Bregman PDMM: primal update}: there exists \(g_i\in N_\mathcal{X}(x_i^{(t+1)})\) such that
\begin{equation}
\begin{aligned}
&-\nu_i^{(t)}+\sum_{j\in\mathcal{V}} P_{ij}\nu_j^{(t)}-\rho\left(\nabla\phi(x_i^{(t+1)})-\nabla\phi(y_i^{(t)})\right)\\
&-\delta_i\left(\nabla\varphi_i(x_i^{(t+1)})-\nabla\varphi_i(x_i^{(t)})\right)-g_i\in \partial f_i(x_i^{(t+1)}).
\end{aligned}
\label{Bregman PDMM primal update: optimality}
\end{equation}
Our goal is to show that as \(t\to\infty\), \(\{\bm{x}^{(t+1)}, \bm{\nu}^{(t+1)}\}\) will satisfy \eqref{KKT: primal} and reduce conditions in \eqref{Bregman PDMM primal update: optimality} to those in \eqref{KKT: dual}. Note that if \(\bm{x}^{(t+1)}=(P\otimes I_n)\bm{x}^{(t+1)}\), then \(\bm{\nu}^{(t+1)}=\bm{\nu}^{(t)}\). Therefore, KKT conditions \eqref{KKT} are satisfied by \(\{\bm{x}^{(t+1)}, \bm{\nu}^{(t+1)}\}\) if the following holds
\begin{equation}
\bm{x}^{(t+1)}=(P\otimes I_n)\bm{x}^{(t+1)}, \quad \bm{x}^{(t+1)}=\bm{x}^{(t)}=\bm{y}^{(t)}. 
\label{optimality conditions}
\end{equation} 
Define the residuals of optimality conditions \eqref{optimality conditions} at iteration \(t\) as 
\begin{equation}
\begin{aligned}
&R(t+1)\coloneqq \frac{\gamma}{2}\norm{((I_m-P)\otimes I_n)\bm{x}^{(t+1)}}_2^2\\
&+\sum_{i\in\mathcal{V}} B_\phi(x_i^{(t+1)}, y_i^{(t)})+\sum_{i\in\mathcal{V}} \frac{\delta_i}{\rho} B_{\varphi_i}(x_i^{(t+1)}, x_i^{(t)}),
\end{aligned}
\label{residuals}
\end{equation}
where \(\gamma>0\). Notice that \(R(t+1)=0\) if and only if \eqref{optimality conditions} holds. Hence \(R(t+1)\) is a running distance to KKT conditions in \eqref{KKT}.
Define the Lyapunov function of Algorithm~\ref{Bregman PDMM}, which measures a running distance to optimal primal-dual pair \((\bm{x}^\star, \bm{\nu}^\star)\) as,
\begin{equation}
\begin{aligned}
V(t)\coloneqq &\frac{1}{2\tau\rho}\norm{\bm{\nu}^\star-\bm{\nu}^{(t)}}_2^2 +\sum_{i\in\mathcal{V}}B_\phi(x_i^\star, y_i^{(t)})\\
&+\sum_{i\in\mathcal{V}} \frac{\delta_i}{\rho} B_{\varphi_i}(x_i^\star, x_i^{(t)}).
\end{aligned}
\label{Lyapunov function}
\end{equation}
 
We first establish the global convergence of Algorithm~\ref{Bregman PDMM} by showing that as \(t\to\infty\), \(V(t)\) is monotonically non-increasing and that \(R(t+1)\to 0\) (see \cite{yu2018bregman} for detailed proof).

\begin{theorem}
Suppose that Assumption~\ref{basic assumption} holds. Let the sequence \(\{\bm{y}^{(t)}, \bm{x}^{(t)}, \bm{\nu}^{(t)}\}\) be generated by Algorithm~\ref{Bregman PDMM}. Let \(R(k+1)\) and \(V(k)\) be defined as in \eqref{residuals} and \eqref{Lyapunov function}, respectively. Set
\begin{equation}
\tau\leq \rho(\mu\sigma-\gamma), \quad 0<\gamma<\mu\sigma,
\label{parameter Bregman PDMM}
\end{equation}
where \(\sigma=\min\{1, n^{\frac{2}{p}-1}\}\). Then
\begin{equation}
V(t)-V(t+1)\geq R(t+1),
\label{telescop series}
\end{equation}
As \(t\to\infty\), \(R(t+1)\) converges to zero, and \(\{\bm{x}^{(t)}, \bm{\nu}^{(t)}\}\) converges to a point that satisfy KKT conditions \eqref{KKT}. 
\label{theorem Bregman PDMM global convergence}
\end{theorem}

The sketch of the proof is as follows. First apply inequality \eqref{definition of subgradient} at \(\bm{x}^{(t+1)}\) and \(\bm{x}^\star\), which yields a non-negative inner product. Then use identity \eqref{3-point property} to break this inner product into three parts, each of which contributes to \(V(t), V(t+1),\) and \(R(t+1)\), respectively. Lemma~\ref{lemma Pythagorean}, entry (c) and (d) in Assumption~\ref{basic assumption}, together with parameter setting in \eqref{parameter Bregman PDMM} ensures that intermediate terms cancel each other, and finally we reach \eqref{telescop series}. Summing up \eqref{telescop series} from \(t=0\) to \(t=\infty\), we have \(\sum_{t=0}^\infty R(t+1)=V(0)-V(\infty)\leq V(0)\). Therefore, as \(t\to\infty\), we must have \(R(t+1)\to 0\), which implies that \(\{\bm{x}^{(t)}, \bm{\nu}^{(t)}\}\) satisfy \eqref{KKT} in the limit. 

In general, \eqref{parameter Bregman PDMM} implies that as \(p\) increases, step size \(\tau\) needs to decrease. See \cite[Remark 1]{wang2014bregman} for details.

The following theorem establishes the \(O(1/T)\) convergence rate of Algorithm~\ref{Bregman PDMM} in an ergodic sense via the Jensen's inequality.

\begin{theorem}
Suppose that Assumption~\ref{basic assumption} holds. Let the sequence \(\{\bm{y}^{(t)}, \bm{x}^{(t)}, \bm{\nu}^{(t)}\}\) be generated by Algorithm~\ref{Bregman PDMM}. Let \(V(k)\) be defined as in \eqref{Lyapunov function}, \(\mu, \tau, \rho, \gamma\) satisfy \eqref{parameter Bregman PDMM}, \(\bm{\nu}^{(0)}=0\) and \(\bar{\bm{x}}^{(T)}=\frac{1}{T}\sum_{t=1}^{T}\bm{x}^{(t)}\). Then
\begin{subequations}
\begin{align}
&\begin{aligned}
&\sum_{i\in\mathcal{V}} f_i(\bar{x}^{T}_i)-\sum_{i\in\mathcal{V}} f_i(x^\star_i)\\
&\leq \frac{1}{T}\left(\rho\sum_{i\in\mathcal{V}} B_\phi(x^\star_i, y_i^{(0)})+\sum_{i\in\mathcal{V}} \delta_i B_{\varphi_i}(x^\star_i, x_i^{(0)})\right),
\end{aligned}\label{Bregman PDMM function value 1/T}\\
&\frac{1}{2}\norm{((I_m-P)\otimes I_n)\bar{\bm{x}}^{(T)}}_2^2\leq \frac{V(0)}{\gamma T}.\label{Bregman PDMM residual 1/T}
\end{align}
\label{Bregman PDMM 1/T convergence}
\end{subequations}
\label{theorem Bregman PDMM 1/T convergence}
\end{theorem}

Theorem~\ref{theorem Bregman PDMM 1/T convergence} shows that the complexity bound of Algorithm~\ref{Bregman PDMM} with respect to dimensionality \(n\) is determined by the Bregman divergence term. The following corollary gives an example where, with a properly chosen Bregman divergence, Algorithm~\ref{Bregman PDMM} outperforms Algorithm~\ref{PDMM} by a factor of \(O(n/\ln n)\) \cite[Remark 2]{wang2014bregman}.

\begin{corollary}
Suppose that assumption~\ref{basic assumption} holds. Suppose \(\norm{g_i}^2_2\leq M_0\) for all \(g_i\in\partial f_i(x^\star_i)\) and \(i\in\mathcal{V}\), where \(M_0\in\reals_+\). Let the sequence \(\{\bm{y}^{(t)}, \bm{x}^{(t)}, \bm{\nu}^{(t)}\}\) be generated by Algorithm~\ref{Bregman PDMM}. Let \(\gamma=1/4\), \(\tau=\rho/2\), \(\delta_{\max}=\max_i \delta_i\), \(\bm{\nu}^{(0)}=0\), \(\bm{x}^{(0)}=\mathbf{1}_m\otimes(\frac{1}{n}\mathbf{1}_n)\) and \(\bar{\bm{x}}^{(T)}=\frac{1}{T}\sum_{t=1}^{T}\bm{x}^{(t)}\), \(\mathcal{X}\) be the probability simplex, \(\phi\) and \(\varphi_i\) be the negative entropy function, then 
\begin{subequations}
\begin{align}
&\sum_{i\in\mathcal{V}} f_i(\bar{x}^{(T)}_i)-\sum_{i\in\mathcal{V}} f_i(x^\star_i)\leq \frac{m(\rho+\delta_{\max})\ln n}{T}, \label{cor1: objective function value}\\
&\begin{aligned}
&\frac{1}{2}\norm{((I_m-P)\otimes I_n)\bar{\bm{x}}^{(T)}}_2^2\\
\leq &\frac{4m M_0}{ \rho^2(1-\lambda_2(P))^2T}+\frac{ 4m(\rho+\delta_{\max})\ln n}{\rho T}.
\end{aligned}\label{cor1:consensus residual}
\end{align}
\label{cor1: eqn1}
\end{subequations}
\label{corollary 1}
\end{corollary} 
Observe that \eqref{cor1:consensus residual} implies that the convergence bounds on consensus residual can be tightened by designing \(\lambda_2(P)\), which can be achieved efficiently via convex optimization \cite{boyd2004fastest}. 

\section{Numerical examples}\label{numerical examples}
In this section, we present numerical examples to demonstrate the performance of Algorithm~\ref{Bregman PDMM}.  Consider the following special case of \eqref{distributed optimization problem} defined over graph \(\mathcal{G}=(\mathcal{V}, \mathcal{E})\): 
\begin{equation}
\underset{u\in\mathcal{X}}{\mbox{minimize}}  \sum_{i=1}^m \left\langle c_i, u\right\rangle,
\label{experiment problem}
\end{equation}
where \(\mathcal{X}\) is the probability simplex. Such problems have potential applications in, for example, policy design in multi-agent decision making \cite{el2016convex, zhang2016decision}.

We use Algorithm~\ref{PDMM} as benchmark since it includes other popular variants of distributed ADMM \cite{iutzeler2013asynchronous, ling2013decentralized, shi2014linear, chang2015multi} as special cases. Compared to Algorithm~\ref{PDMM} which needs efficient Euclidean projection onto probability simplex \cite{duchi2008efficient}, Algorithm~\ref{Bregman PDMM} can solve \eqref{experiment problem} with closed-form updates suited for massive parallelism \cite{wang2014bregman}.

We compare the performance of Algorithm~\ref{Bregman PDMM} with Algorithm~\ref{PDMM} on problem \eqref{experiment problem}, where entries in \(\{c_1, \ldots, c_m\}\) are sampled from standard normal distribution, graph \(\mathcal{G}\) is randomly generated with edge probability \(0.2\) \cite[p.~90]{mesbahi2010graph}. We use the following parameter setting: \(\rho=1\), \(\tau=1/2\), \(\delta_i=0\) for all \(i\in\mathcal{V}\), \(\phi\) is the negative entropy function. We demonstrate the convergence described by \eqref{cor1: eqn1} in Figure~\ref{experiment}. We observe that Algorithm~\ref{Bregman PDMM} significantly outperforms Algorithm~\ref{PDMM}, especially for large scale problem, and optimizing \(\lambda_2(P)\) further accelerates convergence considerably. 
\begin{figure}
\centering
     \subfloat[\(m=20, n=1000\)\label{exp1}]{%
       \includegraphics[width=0.23\textwidth]{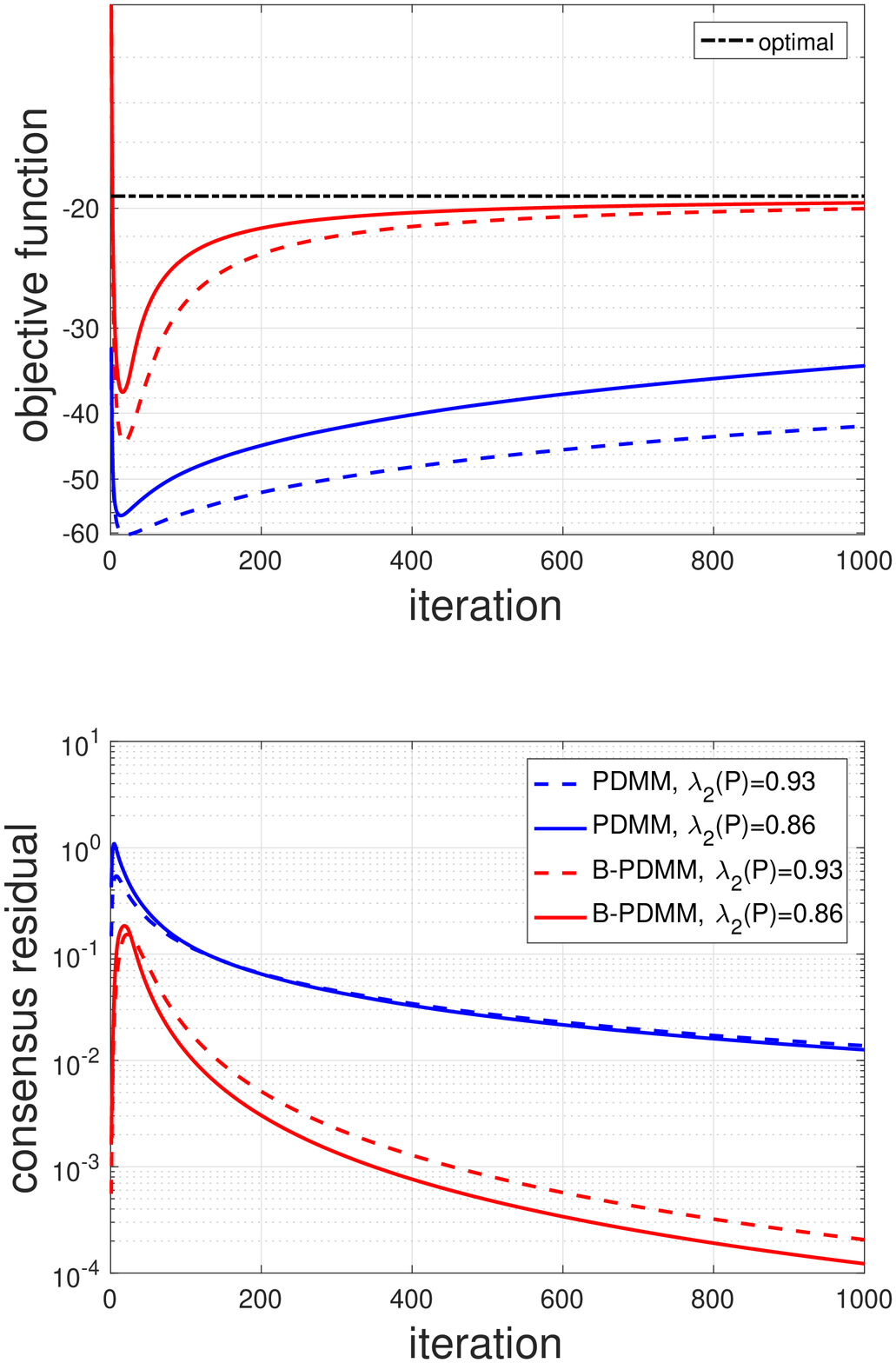}
     }
     \hfill
     \subfloat[\(m=100, n=10000\)\label{exp2}]{%
       \includegraphics[width=0.23\textwidth]{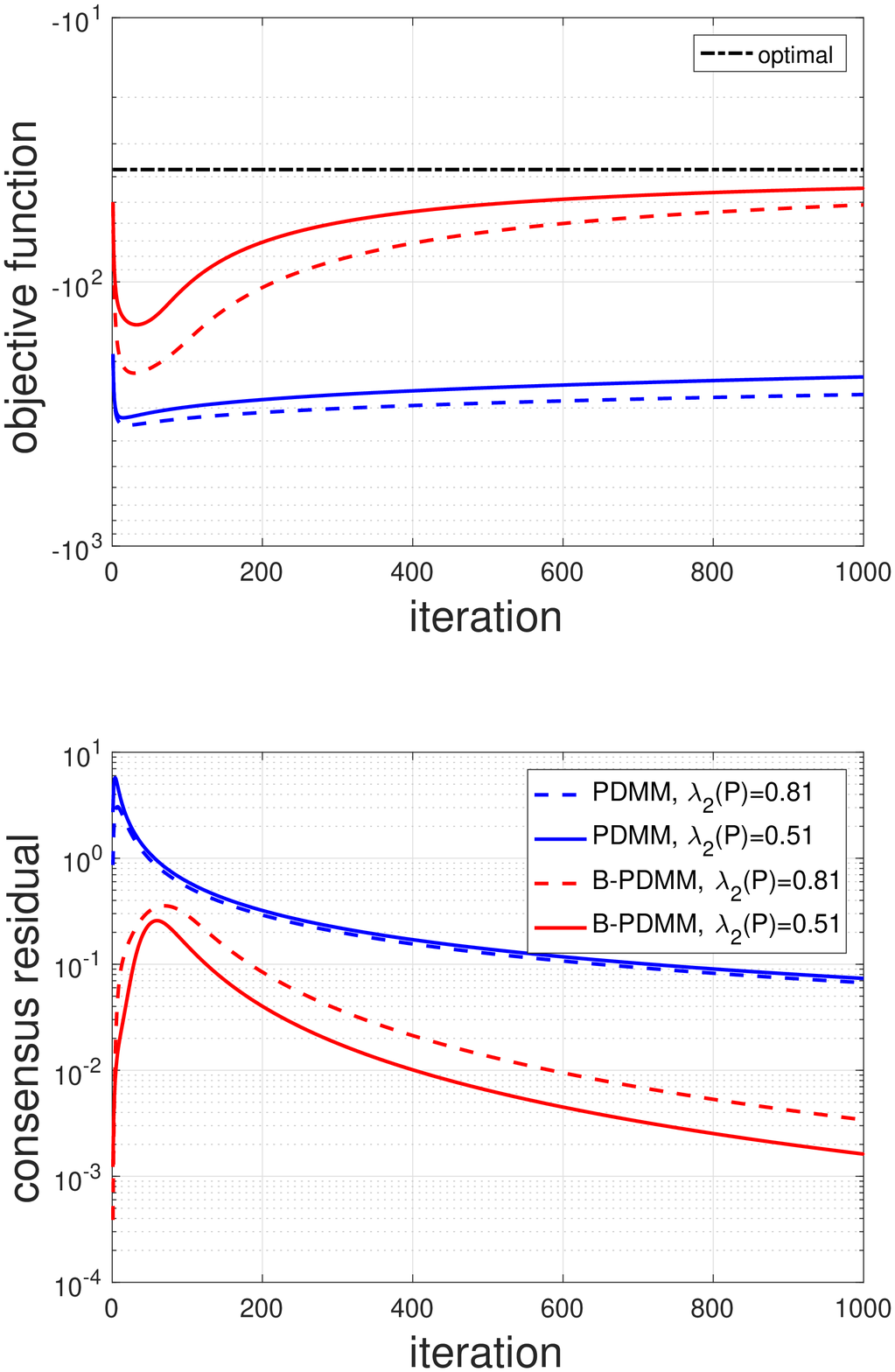}
     }
     \caption{Comparison of Bregman PDMM and PDMM with different \(\lambda_2(P)\).}
     \label{experiment}
\end{figure}

\section{Conclusions}\label{conclusion}
In order to solve distributed optimization over a graph, we generalize PDMM to Bregman PDMM based on mirror averaging. The global convergence and iteration complexity of Bregman PDMM are established, along with its improvement over PDMM. We can further enhance its performance by designing the averaging matrix. Future work directions include the variants of the proposed algorithm for asynchronous and stochastic updates, time-varying graphs, and applications in multi-agent decision making. 





\section*{ACKNOWLEDGMENT}
The authors would like to thank De Meng and Maryam Fazel for many helpful discussions and suggestions. The anonymous reviewers are gratefully acknowledged. 

\bibliographystyle{IEEEtran}
\bibliography{IEEEabrv,reference}

\newpage
\section*{APPENDIX}

We will use the following Lemma.
\begin{lemma}
If \(P\in\reals^{m\times m}\) is symmetric, stochastic and positive semi-definite, then 
\begin{equation}
\begin{aligned}
&\norm{((I_m-P)\otimes I_n)\bm{u}}_2^2\leq \frac{1}{\sigma}\sum_{i, j\in\mathcal{V}}P_{ij}\norm{v_i - u_j}_p^2
\end{aligned}
\label{lemma residual & variance: eqn1}
\end{equation}
where \(\norm{\cdot}_p\) denote \(l_p\) norm and \(\sigma=\min\{1, n^{\frac{2}{p}-1}\}\), \(\bm{u}=[u_1^\top, \cdots, u_m^\top]^\top, \bm{v}=[v_1^\top, \cdots, v_m^\top]^\top\in\mathcal{X}^m\), \(\mathcal{X}\) is a closed convex set.
\label{lemma residual & variance}
\end{lemma}
\begin{proof}
First, observe that if \(P\) is symmetric, stochastic, irreducible and positive semi-definite,  \(P-P^2\) is positive semi-definite \cite[Theorem 8.4.4]{horn2012matrix}. Since \(P\mathbf{1}_m=P^\top\mathbf{1}_m=\mathbf{1}_m\), we can show the following
\begin{equation*}
\begin{aligned}
&\sum_{i, j\in\mathcal{V}}P_{ij}\norm{\sum_{k\in\mathcal{V}} P_{ik}u_k-u_j}_2^2\\
= & \norm{\bm{u}}_2^2-\norm{(P\otimes I_n) \bm{u}}_2^2\\
\geq &\norm{\bm{u}}_2^2-\norm{(P\otimes I_n)\bm{u}}_2^2 - 2\langle \bm{u}, ((P-P^2)\otimes I_n)\bm{u}\rangle\\
=&\norm{((I_m-P)\otimes I_n)\bm{u}}_2^2
\end{aligned}
\end{equation*} 
Hence \eqref{lemma residual & variance: eqn1} holds due to the fact that 
\begin{equation*}
\sum_{k\in\mathcal{V}} P_{ik}u_k = \underset{w\in\mathcal{X}}{\argmin}\sum_{j\in\mathcal{V}} P_{ij}\norm{w-u_j}_2^2,
\end{equation*}
for all \(i\in\mathcal{V}\), and that \(\norm{w}_2^2\leq 1/\sigma\norm{w}_p^2\) for all \(w\in\reals^n\) where \(\sigma = \min\{1, n^{\frac{2}{p}-1}\}\) \cite[Theorem~1]{wang2014bregman}.
\end{proof}
\subsection{Lemma~\ref{lemma Pythagorean} }
\begin{proof}
Since \(P_{ij}=0\) if \(j\notin \mathcal{N}(i)\), the optimality condition for \eqref{mirror Markov mixing dynamics} can be written as follows: for any \(u\in\mathcal{X}\), we have 
\begin{equation*}
\sum_{j\in\mathcal{V}}P_{ij}\langle \nabla\phi(y_i^{(t)})-\nabla \phi(x_j^{(t)}), u-y_i^{(t)}\rangle\geq 0
\end{equation*}
Using three point property \eqref{3-point property}, we have
\begin{equation}
\begin{aligned}
&\sum_{j\in\mathcal{V}}P_{ij}B_\phi(u, x_j^{(t)})-\sum_{j\in\mathcal{V}}P_{ij}B_\phi(u, y_i^{(t)}) \\
\geq&\sum_{j\in\mathcal{V}}P_{ij}B_\phi(y_i^{(t)}, x_j^{(t)})
\end{aligned}
\label{lemma Pythagoream: eqn2}
\end{equation}
Summing \eqref{lemma Pythagoream: eqn2} over all \(i\) completes the proof.
\end{proof}

\subsection{Theorem~\ref{theorem Bregman PDMM global convergence} }
\begin{proof}
Since \(P\) is irreducible, \(\bm{x}^\star\) satisfy \eqref{KKT: primal} if and only if there exists \(x^\star\in\mathcal{X}\) such that \(\bm{x}^\star=\mathbf{1}_m\otimes x^\star\). Substitute \eqref{Bregman PDMM primal update: optimality} into \eqref{definition of subgradient} we have: there exists \(g_i\in N_\mathcal{X}(x_i^{(t+1)})\) for all \(i\) such that
\begin{equation}
\begin{aligned}
&\sum_{i\in\mathcal{V}} f_i(x_i^{(t+1)})-\sum_{i\in\mathcal{V}} f_i(x^\star)\\
\leq &\langle -\bm{\nu}^{(t)},((I_m-P)\otimes I_n)\bm{x}^{(t+1)}  \rangle\\
& + \rho\sum_{i\in\mathcal{V}}\langle \nabla\phi(x_i^{(t+1)})-\nabla\phi(y_i^{(t)}), x^\star-x_i^{(t+1)}\rangle\\
& +  \sum_{i\in\mathcal{V}}\delta_i\langle \nabla\varphi_i(x_i^{(t+1)})-\nabla\varphi_i(x_i^{(t)}), x^\star-x_i^{(t+1)}\rangle\\
& - \sum_{i\in\mathcal{V}}\langle g_i, x_i^{(t+1)}-x^\star\rangle\\
\overset{\eqref{definition of normal cone}}{\leq} &\langle -\bm{\nu}^{(t)},((I_m-P)\otimes I_n)\bm{x}^{(t+1)}  \rangle\\
& + \rho\sum_{i\in\mathcal{V}}\langle \nabla\phi(x_i^{(t+1)})-\nabla\phi(y_i^{(t)}), x^\star-x_i^{(t+1)}\rangle\\
& +  \sum_{i\in\mathcal{V}}\delta_i\langle \nabla\varphi_i(x_i^{(t+1)})-\nabla\varphi_i(x_i^{(t)}), x^\star-x_i^{(t+1)}\rangle\\
\overset{\eqref{3-point property}}{\leq} & \langle -\bm{\nu}^{(t)},((I_m-P)\otimes I_n)\bm{x}^{(t+1)}  \rangle\\
&+ \rho \sum_{i\in\mathcal{V}} B_\phi(x^\star, y_i^{(t)}) +  \sum_{i\in\mathcal{V}} \delta_i B_{\varphi_i}(x^\star, x_i^{(t)})\\
&-\rho \sum_{i\in\mathcal{V}}B_\phi(x^\star, x_i^{(t+1)}) - \sum_{i\in\mathcal{V}}\delta_i B_{\varphi_i}(x^\star, x_i^{(t+1)})\\ 
& -\rho \sum_{i\in\mathcal{V}}B_\phi(x_i^{(t+1)}, y_i^{(t)}) -  \sum_{i\in\mathcal{V}}\delta_i B_{\varphi_i}(x_i^{(t+1)}, x_i^{(t)})\\
\overset{\eqref{lemma Pythagorean: eqn1} }{\leq} & \langle -\bm{\nu}^{(t)},((I_m-P)\otimes I_n)\bm{x}^{(t+1)}  \rangle\\
& + \rho \sum_{i\in\mathcal{V}} B_\phi(x^\star, y_i^{(t)}) +  \sum_{i\in\mathcal{V}} \delta_i B_{\varphi_i}(x^\star, x_i^{(t)})\\
& -\rho \sum_{i\in\mathcal{V}}B_\phi(x^\star, y_i^{(t+1)}) -  \sum_{i\in\mathcal{V}}\delta_i B_{\varphi_i}(x^\star, x_i^{(t+1)})\\
& -\rho \sum_{i\in\mathcal{V}}B_\phi(x_i^{(t+1)}, y_i^{(t)}) -  \sum_{i\in\mathcal{V}}\delta_i B_{\varphi_i}(x_i^{(t+1)}, x_i^{(t)})\\ 
& -\rho\sum_{i, j\in\mathcal{V}}P_{ij}B_\phi(y_i^{(t+1)}, x_j^{(t+1)})\\
\overset{\eqref{phi strong convexity} }{\leq}& \langle -\bm{\nu}^{(t)},((I_m-P)\otimes I_n)\bm{x}^{(t+1)}  \rangle\\
& + \rho \sum_{i\in\mathcal{V}} B_\phi(x^\star, y_i^{(t)}) +  \sum_{i\in\mathcal{V}} \delta_i B_{\varphi_i}(x^\star, x_i^{(t)})\\
& -\rho \sum_{i\in\mathcal{V}}B_\phi(x^\star, y_i^{(t+1)}) -  \sum_{i\in\mathcal{V}}\delta_i B_{\varphi_i}(x^\star, x_i^{(t+1)})\\
& -\rho \sum_{i\in\mathcal{V}}B_\phi(x_i^{(t+1)}, y_i^{(t)}) -  \sum_{i\in\mathcal{V}}\delta_i B_{\varphi_i}(x_i^{(t+1)}, x_i^{(t)})\\ 
& -\frac{\rho\mu}{2}\sum_{i, j\in\mathcal{V}}P_{ij}\norm{y_i^{(t+1)} - x_j^{(t+1)}}_p^2
\end{aligned}
\label{theorem Bregman PDMM global convergence: eqn1}
\end{equation}

Similarly, substitute \eqref{KKT: dual} into \eqref{definition of subgradient} we have
\begin{equation}
\begin{aligned}
&\sum_{i\in\mathcal{V}} f_i(x^\star)-\sum_{i\in\mathcal{V}} f_i(x_i^{(t+1)})\\
\leq & \langle \bm{\nu}^\star, ((I_m-P)\otimes I_n)\bm{x}^{(t+1)}\rangle
\end{aligned}
\label{theorem Bregman PDMM global convergence: eqn2}
\end{equation}
Notice if we let \(\phi(u)=\frac{1}{2}\norm{u}_2^2\) in \eqref{3-point property}, we can show the following using \eqref{Bregman PDMM: dual update}.
\begin{equation}
\begin{aligned}
& \langle \bm{\nu}^\star-\bm{\nu}^{(t)}, ((I_m-P)\otimes I_n)\bm{x}^{(t+1)}\rangle\\
=  &\frac{1}{2\tau}\norm{\bm{\nu}^\star-\bm{\nu}^{(t)}}_2^2-\frac{1}{2\tau}\norm{\bm{\nu}^\star-\bm{\nu}^{(t+1)}}_2^2\\
 &+\frac{\tau}{2}\norm{((I_m-P)\otimes I_n)\bm{x}^{(t+1)}}_2^2
\end{aligned}
\label{theorem Bregman PDMM global convergence: eqn3}
\end{equation}
Based on \eqref{theorem Bregman PDMM global convergence: eqn3}, sum up \eqref{theorem Bregman PDMM global convergence: eqn1} and \eqref{theorem Bregman PDMM global convergence: eqn2} we obtain

\begin{equation}
\begin{aligned}
&V(t)-V(t+1)\\
\geq &R(t+1)+\frac{\mu}{2}\sum_{i, j\in\mathcal{V}}P_{ij}\norm{y_i^{(t+1)} - x_j^{(t+1)}}_p^2\\
 &-\frac{\tau/\rho+\gamma}{2}\norm{((I_m-P)\otimes I_n)\bm{x}^{(t+1)}}_2^2\\
  \overset{\eqref{lemma residual & variance: eqn1} }{\geq} &R(t+1)+\frac{\mu}{2}\sum_{i, j\in\mathcal{V}}P_{ij}\norm{y_i^{(t+1)} - x_j^{(t+1)}}_p^2\\
 &-\frac{\tau/\rho+\gamma}{2\sigma}\sum_{i, j\in\mathcal{V}}P_{ij}\norm{y_i^{(t+1)}-x_j^{(t+1)}}_p^2\\
 \overset{\eqref{parameter Bregman PDMM} }{\geq} & R(t+1).
 \end{aligned}
 \label{theorem Bregman PDMM global convergence: eqn4}
\end{equation}

Notice that properties of mirror maps ensured that \(x_i^{(t)}\) will not be achieved on the boundary of \(\mathcal{D}\) for all \(i\in\mathcal{V}\) and \(t\), hence \(V(t)<\infty\) for all \(t\). Sum up \eqref{theorem Bregman PDMM global convergence: eqn4} from \(t=0\) to \(\infty\) we have \(\sum_{t=0}^\infty R(t+1)\leq V(0)\). Since \(R(t+1)\geq 0\), \(R(t+1)\to 0\) as \(t\to \infty\), which completes the proof.
\end{proof}
\subsection{Theorem~\ref{theorem Bregman PDMM 1/T convergence} }
\begin{proof}
Since 
\begin{equation}
\begin{aligned}
& \langle -\bm{\nu}^{(t)}, ((I_m-P)\otimes I_n)\bm{x}^{(t+1)}\rangle\\
 \overset{\eqref{Bregman PDMM: dual update} }{=}  &\frac{1}{2\tau}\norm{\bm{\nu}^{(t)}}_2^2-\frac{1}{2\tau}\norm{\bm{\nu}^{(t+1)}}_2^2\\
 & +\frac{\tau}{2}\norm{((I_m-P)\otimes I_n)\bm{x}^{(t+1)}}_2^2
\end{aligned}
\label{theorem Bregman PDMM 1/T convergence: eqn1}
\end{equation}
Substitute \eqref{theorem Bregman PDMM 1/T convergence: eqn1} into \eqref{theorem Bregman PDMM global convergence: eqn1}, combined with the fact that 
\[\rho\sum_{i\in\mathcal{V}} B_\phi(x_i^{(t+1)}, y_i^{(t)})+\sum_{i\in\mathcal{V}}\delta_iB_{\varphi_i}(x_i^{(t+1)}, x_i^{(t)})\geq 0\] 
we obtain 
\begin{equation}
\begin{aligned}
&\sum_{i\in\mathcal{V}} f_i(x_i^{(t+1)})-\sum_{i\in\mathcal{V}} f_i(x^\star)\\
\overset{\eqref{lemma residual & variance: eqn1}}{\leq} & \frac{1}{2\tau}\norm{\bm{\nu}^{(t)}}_2^2-\frac{1}{2\tau}\norm{\bm{\nu}^{(t+1)}}_2^2\\
&+\rho\sum_{i\in\mathcal{V}} B_\phi(x^\star, y_i^{(t)})-\rho\sum_{i\in\mathcal{V}} B_\phi(x^\star, y_i^{(t+1)})\\
&+\sum_{i\in\mathcal{V}} \delta_iB_{\varphi_i}(x^\star, x_i^{(t)}) -\sum_{i\in\mathcal{V}} \delta_i B_{\varphi_i}(x^\star, x_i^{(t+1)})\\
&-\frac{\mu\rho}{2}\sum_{i, j\in\mathcal{V}}P_{ij}\norm{y_i^{(t+1)} - x_j^{(t+1)}}_p^2\\
&+\frac{\tau}{2\sigma}\sum_{i, j\in\mathcal{V}}P_{ij}\norm{y_i^{(t+1)}-x_j^{(t+1)}}_p^2\\
\overset{\eqref{parameter Bregman PDMM} }{\leq} & \frac{1}{2\tau}\norm{\bm{\nu}^{(t)}}_2^2-\frac{1}{2\tau}\norm{\bm{\nu}^{(t+1)}}_2^2\\
&+\rho\sum_{i\in\mathcal{V}} B_\phi(x^\star, y_i^{(t)})-\rho\sum_{i\in\mathcal{V}} B_\phi(x^\star, y_i^{(t+1)})\\
&+\sum_{i\in\mathcal{V}} \delta_iB_\varphi(x^\star, x_i^{(t)}) -\sum_{i\in\mathcal{V}} \delta_i B_{\varphi_i}(x^\star, x_i^{(t+1)})
\end{aligned}
\label{theorem Bregman PDMM 1/T convergence: eqn2}
\end{equation}
Sum up \eqref{theorem Bregman PDMM 1/T convergence: eqn2} for \(t=0, \ldots, T-1\) and apply Jensen's inequality we have \eqref{Bregman PDMM function value 1/T}, similarly sum up \eqref{theorem Bregman PDMM global convergence: eqn4} for  \(t=0, \ldots, T-1\) and apply Jensen's inequality we have \eqref{Bregman PDMM residual 1/T}.
\end{proof}
\subsection{Corollary~\ref{corollary 1}}
\begin{proof}
The proof is a direction application of Theorem~\ref{theorem Bregman PDMM 1/T convergence} and the fact that \(\bm{\nu}^\star\) is in the range space of matrix \((I_m-P)\otimes I_n\) if \(\bm{\nu}^{(0)}=0\) (See Lemma 1 in \cite{meng2015proximal}).
\end{proof}


%
%
%
%
%

\end{document}